\documentclass[10pt]{article}

\usepackage{amsthm}
\usepackage[ansinew]{inputenc}   
\usepackage[english]{babel}
\usepackage{anysize}
\usepackage{amsfonts}
\usepackage{amssymb}
\usepackage{amsmath}
\usepackage{doc}
\usepackage{exscale}
\usepackage{fontenc}
\usepackage{ifthen}
\usepackage{latexsym}
\usepackage{makeidx}
\usepackage{syntonly}
\usepackage{graphicx}
\usepackage{float}
\usepackage{array}
\usepackage[all]{xy}
\usepackage[usenames,dvipsnames]{color}

\marginsize{2.5cm}{2.5cm}{2.5cm}{2.5cm}

\newtheorem{thm}{Theorem}[section]
\newtheorem{cor}[thm]{Corollary}
\newtheorem{lem}[thm]{Lemma}
\newtheorem{pro}[thm]{Proposition}
\newtheorem{df}[thm]{Definition}
\newtheorem{rmk}[thm]{Remark}

\newtheorem{ex}[thm]{Example}
\newtheorem{lem-df}[thm]{Lemma-Definition}
\newtheorem{conj}[thm]{Conjecture}
\newtheorem{rmk-df}[thm]{Remark-Definition}

\newtheorem*{thm*}{Theorem}

\newcommand{\beq}{\begin{equation}}
\newcommand{\enq}{\end{equation}}
\newcommand{\beqn}{\begin{equation*}}
\newcommand{\enqn}{\end{equation*}}

\newcommand{\ra}{\rightarrow}

\newcommand{\hra}{\hookrightarrow}
\newcommand{\longra}{\longrightarrow}

\newcommand{\longhra}{\ensuremath{\lhook\joinrel\relbar\joinrel\rightarrow}}

\newcommand{\mC}{\mathbb{C}}

\newcommand{\mP}{\mathbb{P}}

\newcommand{\caC}{\mathcal{C}}

\newcommand{\caE}{\mathcal{E}}
\newcommand{\caF}{\mathcal{F}}

\newcommand{\caK}{\mathcal{K}}

\newcommand{\caM}{\mathcal{M}}

\newcommand{\caO}{\mathcal{O}}

\newcommand{\caT}{\mathcal{T}}

\DeclareMathOperator{\Hom}{Hom}

\DeclareMathOperator{\Ext}{Ext}

\DeclareMathOperator{\im}{im}

\DeclareMathOperator{\Alb}{Alb}

\DeclareMathOperator{\Cliff}{Cliff}

\DeclareMathOperator{\Spec}{Spec}

\DeclareMathOperator{\rk}{rk}

\begin{document}

\title{Xiao's conjecture for general fibred surfaces}

\author{Miguel \'Angel Barja, V\'{\i}ctor Gonz\'{a}lez-Alonso and Juan Carlos Naranjo\footnote{During the developement of this work, the first and second authors were supported by the Spanish ``Ministerio de Econom\'{\i}a y Competitividad'' (through the project MTM2012-38122-C03-01/FEDER) and the ``Generalitat de Catalunya'' (through the project 2009-SGR-1284), and the third author was supported by the project MTM2012-38122-C03-02 of the Spanish ``Ministerio de Econom\'{\i}a y Competitividad''. The second author was also supported by the grant FPU-AP2008-01849 of the Spanish ``Ministerio de Educaci\'{o}n'' and by the ERC StG 279723 ``Arithmetic of algebraic surfaces'' (SURFARI).}}

\maketitle

\begin{abstract}
We prove that the genus $g$, the relative irregularity $q_f$ and the Clifford index $c_f$ of a {\em non-isotrivial} fibration $f$ satisfy the inequality $q_f \leq g - c_f$. This gives in particular a proof of Xiao's conjecture for fibrations whose general fibres have maximal Clifford index.
\end{abstract}

\section{Introduction}

\label{sect-intro}

In the classification of smooth algebraic surfaces it is natural to study its possible fibrations over curves, trying to relate the geometry of the surface to the properties of the fibres and the base. In this article we focus on the relations between numerical invariants of a fibration, proving Xiao's conjecture for fibrations whose general fibres have maximal Clifford index.

Let $f: S \ra B$ be a fibration from a compact surface $S$ to a compact curve $B$, (that is, a surjective morphism with connected fibres), and let $F$ be a general (smooth) fibre of $f$. The fibration is called {\em isotrivial} if all the smooth fibres are mutually isomorphic, and it is {\em trivial} if $S$ is birational to $B \times F$ and the given fibration corresponds to the first projection.

We first consider the genus $g$ of $F$ (also called the genus of $f$) and the relative irregularity $q_f = q\left(S\right) - g\left(B\right)$. Beauville showed in its Appendix to \cite{Deb-Beau} that
\beq \label{ineq-Beau}
0 \leq q_f \leq g,
\enq
and the equality $q_f = g$ holds if and only if $f$ is {\em trivial}. As a consequence of the work of Serrano \cite{Ser-Iso}, non-trivial isotrivial fibrations satisfy 
\beq \label{main-ineq}
q_f \leq \frac{g+1}{2}.
\enq

For non-isotrivial fibrations, the only known general upper bound for $q_f$ is
\beq \label{bound-5/6}
q_f \leq \frac{5g+1}{6},
\enq
proven by Xiao in \cite{Xiao-5/6}. However, in his later work \cite{Xiao-P1}, Xiao says literally that ``it is unlikely that this inequality gives the best bound for $q$, since its proof is not very accurate''. By {\em not very accurate} he might mean that his proof uses only properties of the first step of the Harder-Narasimham filtration of the vector bundle $f_*\omega_{S/B}$, and hence the result might be improved by taking into account the complete filtration. In fact, in the same work \cite{Xiao-P1} he proves that, in the special case in which the base is $B \cong \mP^1$, the upper-bound (\ref{main-ineq}) holds for any non-trivial fibration, regardless whether it is isotrivial or not. In view of this result, Xiao conjectured in \cite{Xiao-conj} that the inequality (\ref{main-ineq}) should hold for every non-trivial fibration, and he provided an example attaining the equality. This conjecture was shown to be false by Pirola in \cite{Pir-Xiao}, where he provided a non-isotrivial fibration with fibres of genus $g = 4$ and relative irregularity $q_f = 3 \not\leq \frac{5}{2} = \frac{g+1}{2}$. Recently, Albano and Pirola \cite{Alb-Pir} have obtained more counterexamples with genus $g=6$ and $10$, all of them satisfying
\beqn
q_f = \frac{g}{2}+1 = \frac{g+1}{2} + \frac{1}{2}.
\enqn
The fact that in all known counterexamples the conjecture fails by exactly $\frac{1}{2}$ naturally leads to the following modification.
\begin{conj}[modified Xiao's conjecture] \label{final-conj}
For any non-trivial fibration $f: S \ra B$ one  has
\beqn
q_f \leq \frac{g}{2}+1,
\enqn
or equivalently
\beqn
q_f \leq \left\lceil\frac{g+1}{2}\right\rceil.
\enqn
\end{conj}

Note that for odd values of $g$, the bound in Conjecture \ref{final-conj} is equivalent to the inequality (\ref{main-ineq}) originally conjectured by Xiao. It is worth to note that the original version of the conjecture has been proved for several classes of fibrations. On the one hand, in addition to the counterexample, Pirola proves in \cite{Pir-Xiao} that the conjecture holds if a sort of Abel-Jacobi map (from the base of the fibration to a primitive intermediate Jacobian) is constant. On the other hand, already in 1998 Cai \cite{Cai} proved the conjecture for fibrations whose general fibre is either hyperelliptic or bielliptic.

In this article we prove the following
\begin{thm} \label{main-thm}
Let $f: S \ra B$ be a fibration of genus $g \geq 2$, relative irregularity $q_f$ and Clifford index $c_f$. If $f$ is non-isotrivial, then
\beqn
q_f \leq g-c_f.
\enqn
\end{thm}
The {\em Clifford index} of $f$, $c_f$, was introduced by Konno in \cite{Kon} (Def. 1.1) as the Clifford index of a general fibre. It is in fact the maximum of the Clifford indexes of the smooth fibres, which is attained over a non-empty Zariski-open subset of $B$. The Clifford index has a role in several improvements of the {\em slope-inequality}, as those obtained by Konno himself, and by Barja and Stoppino in \cite{Bar-Sto}.

Note that, as soon as $c_f > \frac{g-1}{6}$, Theorem \ref{main-thm} is an improvement of Xiao's general inequality (\ref{bound-5/6}). In the particular case of maximal Clifford index, we obtain the following
\begin{cor}
If $f$ is not trivial and $c_f$ is maximal, i.e., $c_f = \left\lfloor\frac{g-1}{2}\right\rfloor$, then $q_f \leq \left\lceil\frac{g+1}{2}\right\rceil$.
\end{cor}

Since the set of curves with maximal Clifford index is a Zariski-open subset of the moduli space $\caM_g$ of curves of genus $g$, it makes sense to say that such a fibration is a {\em general fibration}. Hence Theorem \ref{main-thm} can be interpreted as the proof of Conjecture \ref{final-conj} for general fibrations.

In order to prove Theorem \ref{main-thm}, we first need the existence of a supporting divisor for $f$, which is guaranteed by the more general study of families of irregular varieties carried out by the second named author in \cite{Xiao-1}. By a {\em supporting divisor} for $f$, we mean a divisor on $S$ whose restriction to a general fibre supports the corresponding first-order deformation induced by $f$ (see Definitions \ref{df-supp-infinit} and \ref{df-supp-global}). Once this divisor is obtained, we must consider whether its restriction to a general fibre is rigid or not. On the one hand, if it is rigid we can conclude using a structure result also proved in \cite{Xiao-1} (Theorem \ref{thm-struc}). This case can also be handled in an alternative way, with more local flavor and quite similar to Xiao's original method. This different approach is the content of the last section of the paper. On the other hand, if the supporting divisor moves in every fibre, we need a result on the rank of first-order deformations of curves (Theorem \ref{thm-supp-lower-bound}). This result was stated by Ginensky as part of a more general theorem in \cite{Gin}, whose original proof contains some inacuracies. Though the part of the proof we need can be completed and slightly shortened, we have decided to include here a different, much shorter proof, suggested to us by Pirola.

\medskip

\textbf{Acknowledgements:} We would like to thank Prof. Gian Pietro Pirola for the many stimulating discussions around this topic, especially for presenting to us several counterexamples to the original conjecture of Xiao and for suggesting the new proof of the result of Ginensky (Theorem \ref{thm-supp-lower-bound}). We are also grateful to the referees for their suggestions, which helped us to improve the exposition of our results.

\medskip

\textbf{Basic assumptions and notation:} Throughout the whole article, all varieties are assumed to be smooth and defined over $\mC$. Unless otherwise is explicitly said, $f: S \ra B$ will be a fibration (a surjective morphism with connected fibres) from a compact surface $S$ to a compact curve $B$. The {\em genus} of $f$, defined as the genus of any smooth fibre, will be denoted by $g$, and assumed to be at least 2. The {\em relative irregularity} of $f$ is by definition the difference $q_f = q\left(S\right) - g\left(B\right)$. According to Fujita's decomposition theorem (\cite{Fuj1},\cite{Fuj2}), $q_f$ coincides with the rank of the trivial part of the locally free sheaf $f_*\omega_{S/B}$.


\section{Preliminaries}

\label{sect-prel}

We will use some notions about infinitesimal deformations of curves, as well as some results on fibred surfaces developed in the previous work \cite{Xiao-1}.

\subsection{Infinitesimal deformations}

\label{subsect-inf-def}

Let $C$ be a smooth curve of genus $g \geq 2$. A first-order infinitesimal deformation of $C$ is a proper flat morphism $\caC \ra \Delta$ over the spectrum of the dual numbers $\Delta = \Spec \mC\left[\epsilon\right]/\left(\epsilon^2\right)$, such that the special fibre (over $0 = \Spec \mC\left(\epsilon\right)$) is isomorphic to $C$. A first-order infinitesimal deformation is determined (up to isomorphism) by its Kodaira-Spencer class $\xi \in H^1\left(C,T_C\right)$, defined as the extension class of the sequence defining the conormal bundle,
\beq \label{eq-xi-fibra}
0 \longra N_{C/\caC}^{\vee}  \cong T_{\Delta,0}^{\vee} \otimes \caO_C \longra \Omega_{\caC|C}^1 \longra \omega_C \longra 0,
\enq
after choosing an isomorphism $T_{\Delta,0}^{\vee} \cong \mC$. We will assume that the deformation is not trivial, that is $\caC \not \cong C \times \Delta$, or equivalently, $\xi \neq 0$.

Cup-product with $\xi$ gives a map
\beqn
\partial_{\xi} = \cup \, \xi: H^0\left(C,\omega_C\right) \longra H^1\left(C,\caO_C\right)
\enqn
that coincides with the connecting homomorphism in the exact sequence of cohomology of (\ref{eq-xi-fibra}).

\begin{df} \label{df-rk-xi}
The {\em rank} of $\xi$ is
\beqn
\rk \xi = \rk \partial_{\xi}.
\enqn
\end{df}

If $C$ is non-hyperelliptic, the map $H^1\left(C,T_C\right) \ra \Hom\left(H^0\left(C,\omega_C\right),H^1\left(C,\caO_C\right)\right)$ given by $\xi \mapsto \partial_{\xi}$ is injective, hence no information is lost when considering $\partial_{\xi}$ instead of $\xi$. However, if $C$ is hyperelliptic, the above map is not injective, and we may have $\rk \xi = 0$ even if $\xi \neq 0$. This exception is a manifestation of the failure of the infinitesimal Torelli Theorem for hyperelliptic curves.

From now on, until the end of the section, $D$ will denote an effective divisor on $C$ of degree $d$. We will also denote by $r = r(D) = h^0\left(C,\caO_C\left(D\right)\right)-1$ the dimension of its complete linear series.

\begin{df} \label{df-supp-infinit}
The deformation $\xi$ is {\em supported on} $D$ if and only if
\beqn
\xi \in \ker\left(H^1\left(C,T_C\right) \longra H^1\left(C,T_C\left(D\right)\right)\right),
\enqn
where the map is induced by the injection of line bundles $T_C \stackrel{+D}{\longra} T_C\left(D\right)$. Furthermore, if $\xi$ is not supported on any strictly smaller effective divisor $D' < D$, we say that $\xi$ is {\em minimally supported on} $D$.
\end{df}

As far as we are aware, the notion of supporting divisor was introduced in \cite{ColPir}, while the minimality was first considered in \cite{Gin}. The use of the word ``support'' has two motivations. On the one hand, $\xi$ is supported on $D$ if and only if it is the image of a Laurent tail of a meromorphic section $\eta \in H^0\left(D,T_C\left(D\right)_{|D}\right)$, which is obviously supported on $D$. On the other hand, $\xi$ is supported on $D$ if and only if, in the bicanonical space of $C$, the line $\mC\left\langle\xi\right\rangle$ corresponds to a point in the span of $D$.

If $D$ has the smallest degree among the divisors supporting $\xi$, then $\xi$ is minimally supported on $D$, but not conversely. Indeed, $\xi$ being minimally supported on $D$ means that it is not possible to remove some point of $D$ and still support $\xi$, but there is no reason for $D$ to have minimal degree.

One could equivalently define $\xi$ to be supported on the divisor $D$ if and only if the top row in the following pull-back diagram is split. 
\beq \label{diag-split}
\xymatrix{
\xi_D : \quad 0 \ar[r] & N_{C/\caC}^{\vee} \ar[rr] \ar@{=}[d] & & \caF_D \ar[rr] \ar@{^(->}[]+<0mm,-3mm>;[d] & & \omega_C\left(-D\right) \ar[r] \ar@{^(->}[]+<0mm,-3mm>;[d] \ar@/_/@{-->}[ll] & 0 \\
\xi : \quad 0 \ar[r] & N_{C/\caC}^{\vee} \ar[rr] & & \Omega_{\caC|C}^1 \ar[rr] & & \omega_C \ar[r] & 0
}
\enq
Indeed, the map $H^1\left(C,T_C\right) \ra H^1\left(C,T_C\left(D\right)\right)$ is naturally identified with the pull-back of extensions $\Ext^1_{\caO_C}\left(\omega_C,\caO_C\right) \ra \Ext^1_{\caO_C}\left(\omega_C\left(-D\right),\caO_C\right)$.

The following is a first relation between the rank of a deformation and the invariants of a supporting divisor.

\begin{lem} \label{lem-supp-inclu}
Suppose $\xi$ is supported on $D$. Then $H^0\left(C,\omega_C\left(-D\right)\right) \subseteq \ker \partial_{\xi}$. Hence
\beqn
\rk \xi \leq \deg D - r\left(D\right).
\enqn
\end{lem}
\begin{proof}
The fact that $\xi_D$ is split implies that all the sections of $\omega_C\left(-D\right)$ lift to sections of $\Omega_{\caC|C}^1$, and hence belong to the kernel of $\partial_{\xi}$. The inequality is an easy consequence of Riemann-Roch, because
\beqn
\rk \xi = g - \dim \ker \partial_{\xi} \leq g - h^0\left(C,\omega_C\left(-D\right)\right) = g - \left(r\left(D\right)-d+g\right) = d - r\left(D\right).
\enqn
\end{proof}

We will need also a lower-bound on $\rk \xi$, which was first proved by Ginensky in \cite{Gin}. We include here a different (and shorter) proof, suggested to us by Pirola. Recall that the Clifford index of any divisor $D$ is defined as
\beqn
\Cliff\left(D\right) = \deg D - 2 r\left(D\right).
\enqn
Recall also that the Clifford index of the curve $C$ is
\beqn
\Cliff\left(C\right) = \min\left\{\Cliff\left(D\right)\,\vert\,h^0\left(C,\caO_C\left(D\right)\right),h^1\left(C,\caO_C\left(D\right)\right) \geq 2\right\}.
\enqn

\begin{thm} \label{thm-supp-lower-bound}
If $\xi$ is minimally supported on $D$, then
\beqn
\rk \xi \geq \deg D - 2r\left(D\right) = \Cliff\left(D\right).
\enqn
\end{thm}
\begin{proof}
Since $\xi$ is supported on $D$, the inclusion $\omega_C\left(-D\right) \hra \omega_C$ factors through $\iota_D: \omega_C\left(-D\right) \hra \Omega_{\caC|C}^1$.

\textit{Claim:} If $D$ supports $\xi$ minimally, the cokernel $\caK_D$ of $\iota_D$ is torsion-free.

Assuming the claim, the proof can be completed as follows. On the one hand, comparing determinants one has $\caK_D \cong \caO_C\left(D\right)$, giving the exact sequence of sheaves
\beqn
0 \longra \omega_C\left(-D\right) \longra \Omega_{\caC|C}^1 \longra \caO_C\left(D\right) \longra 0,
\enqn
from which the inequality
\beq \label{eq-a}
h^0\left(C,\Omega_{\caC|C}^1\right) \leq h^0\left(C,\omega_C\left(-D\right)\right) + h^0\left(C,\caO_C\left(D\right)\right)
\enq
follows. On the other hand, from the exact sequence of cohomology of $\xi$, one gets
\beq \label{eq-b}
g - \rk \xi = \dim \ker \partial_{\xi} = h^0\left(C,\Omega_{\caC|C}^1\right) - 1.
\enq
Combining inequality (\ref{eq-a}) and equality (\ref{eq-b}) with Riemann-Roch, one finally obtains
\beqn
g - \rk \xi \leq h^0\left(C,\omega_C\left(-D\right)\right) + h^0\left(C,\caO_C\left(D\right)\right) - 1 = 2 r\left(D\right) - \deg D + g.
\enqn

\textit{Proof of the claim:}
We will show in fact that if $\caK_D$ has torsion, then $D$ does not support $\xi$ minimally. Indeed, if $\caT \neq 0$ is the torsion subsheaf of $\caK_D$, there is a line bundle $\caM$ such that
\beqn
0 \longra \caM \longra \Omega_{\caC|C}^1 \longra \caK_D/\caT \longra 0 \quad \text{and} \quad 0 \longra \omega_C\left(-D\right) \longra \caM \longra \caT \longra 0.
\enqn
The image of the composition $\phi: \caM \hra \Omega_{\caC|C}^1 \ra \omega_C$ contains $\omega_C\left(-D\right)$ by construction. Therefore $\phi$ is injective, and gives an isomorphism $\caM \cong \omega_C\left(-E\right)$ for some $0 \leq E < D$. Since $\caM \hra \omega_C$ factors through $\Omega_{\caC|C}^1$, this contradicts the minimality of $D$, as wanted.
\end{proof}

\begin{ex}
There exist examples such that $\rk\xi = \Cliff\left(D\right)$. Let for example $C$ be a general fibre of a fibration $f$ such that $q_f = \frac{g}{2}+1 > \frac{g+1}{2}$ (as those constructed in \cite{Alb-Pir}), and let $\xi$ be its first-order deformation induced by $f$. As a consequence of the proof of Theorem \ref{main-thm}, any divisor $D$ supporting $\xi$ must be movable and will satisfy $\rk \xi = \Cliff\left(D\right)$.
Furthermore, such examples are not only infinitesimal deformations, but actual families of curves with the same property. We have not been able to find more explicit examples. Although Ginensky claims to classify these cases in \cite{Gin} Theorem 2.5, we cannot follow this part of his proof.
\end{ex}

\subsection{Fibred surfaces}

\label{subsect-fibr-surf}

We now recall some results on fibred surfaces proved in the previous work \cite{Xiao-1} of the second named author. Let $f: S \ra B$ be a fibration of genus $g$ and relative irregularity $q_f$. We say that $f$ is {\em isotrivial} if all the smooth fibres are isomorphic. For any smooth fibre $C_b$, the kernel of the restriction map $r_b: H^0\left(S,\Omega_S^1\right) \ra H^0\left(C_b,\omega_{C_b}\right)$ is exactly $f^*H^0\left(B,\omega_B\right)$. Therefore, there is an injection
\beq \label{inj-V_f}
V_f := H^0\left(S,\Omega_S^1\right)/f^*H^0\left(B,\omega_B\right) \longhra H^0\left(C_b,\omega_{C_b}\right)
\enq
which implies the inequality $q_f \leq g$.

\begin{rmk} \label{rmk-kernel}
Although it is well known, we would like to sketch a proof of the equality $\ker r_b = f^*H^0\left(B,\omega_B\right)$, using a construction that will appear again in the last section. The ideas are actually taken from \cite{Deb-Beau} and \cite{Pir-Xiao}. The basic fact is that the image of the Jacobian of a smooth fibre $C_b$ in $\Alb\left(S\right)$ is independent of $b \in B$ (up to translation), and is precisely $A = \ker\left(\Alb\left(S\right)\ra J\left(B\right)\right)$. We have therefore an exact sequence of Abelian varieties
$$J\left(C_b\right) \longra \Alb\left(S\right) \longra J\left(B\right) \longra 0,$$
and looking at the maps on the cotangent spaces at the origin we obtain the exact sequence of vector spaces
$$0 \longra H^0\left(B,\omega_B\right) \stackrel{f^*}{\longra} H^0\left(S,\Omega_S^1\right) \stackrel{r_b}{\longra} H^0\left(C_b,\omega_{C_b}\right),$$
hence the wanted inclusion
$$V_f:=H^0\left(S,\Omega_S^1\right)/f^*H^0\left(B,\omega_B\right) = H^0\left(A,\Omega_A^1\right) \hookrightarrow H^0\left(C_b,\omega_{C_b}\right).$$
\end{rmk}

For any finite map $\pi: B' \ra B$, let $S' = \widetilde{S \times_B B'}$ be the minimal desingularization of the fibred product, and $f': S' \ra B'$ the induced fibration. The fibres of $f'$ obviously have the same genus as the fibres of $f$, but for the relative irregularity only the inequality $q_{f'} \geq q_f$ can be proved (which might be strict). Indeed, for any $b \in B$ where $\pi$ is not ramified, the injection (\ref{inj-V_f}) factors as
\beqn
V_f \longra V_{f'} \longhra H^0\left(C_b,\omega_{C_b}\right),
\enqn
which forces the first map to be injective, and hence $q_f \leq q_{f'}$.

For any smooth fibre $C_b$, denote by $\xi_b \in H^1\left(C_b,T_{C_b}\right)$ the class of the first order deformation induced by $f$. 

\begin{df}[\cite{Xiao-1} Definition 2.3] \label{df-supp-global}
Let $D \subset S$ be an effective divisor. The fibration $f$ is {\em supported on} $D$ if for a general $b \in B$, $\xi_b$ is supported on $D_{|C_b}$.
\end{df}

Note that this definition is local around the smooth fibres. As a consequence, if $f$ is supported on $D$ and we perform a change of base $\pi: B' \ra B$ as above, then $f'$ is supported on $\pi'^*D \subset S'$ (where $\pi': S' \ra S$ is the induced map between the surfaces).

The existence of supporting divisors is investigated in \cite{Xiao-1}. For our current purposes, the most useful result is the following.

\begin{thm}[\cite{Xiao-1} Corollary 3.2] \label{thm-supp-div}
If $q_f > \frac{g+1}{2}$, then after a base change $B' \ra B$, there is a divisor $D \subset S'$ supporting $f': S' \ra B'$ and such that $D \cdot C_b < 2g-2$ for any fibre $C_b$. Furthermore, if $f$ is relatively minimal with reduced fibres, then $D \cdot C \leq 2g\left(C\right)-2-C^2$ for any component $C$ of a fibre.
\end{thm}

The proof of this result uses {\em adjoint images}, whose definitions and main properties will be recalled in the last section. Roughly speaking, the condition $q_f > \frac{g+1}{2}$ implies that locally around every smooth fibre $C_b$ there are two linearly independent holomorphic forms with vanishing adjoint image, and the (relative) base divisor of the linear system they span is precisely the supporting divisor $D$. Furthermore, according to the forthcoming Remark \ref{rmk-ramif-alb}, the divisor $D$ in Theorem \ref{thm-supp-div} contains the components of the ramification divisor of the Albanese map of $S$ that are not contained in fibres.

The proof of Theorem \ref{main-thm} splits into two cases, depending on whether the restriction of a supporting divisor to a general fibre is rigid or not. In the rigid case, there is a strong result on the structure of the fibred surface which will be very useful:

\begin{thm}[\cite{Xiao-1} Theorem 2.1] \label{thm-struc}
Let $S$ be a compact surface, and $f: S \ra B$ a stable fibration by curves of genus $g$ and relative irregularity $q_f =q\left(S\right) - g\left(B\right) \geq 2$. Suppose $f$ is supported on an effective divisor $D$ without components contained in fibres. Suppose also that $D \cdot C \leq 2g\left(C\right)-2-C^2$ for any component $C$ of a fibre, and that $h^0\left(C_b,\caO_{C_b}\left(D_{|C_b}\right)\right) = 1$ for some smooth fibre $C_b$. Then there is another fibration $h: S \ra B'$ over a curve of genus $g\left(B'\right) = q_f$. In particular $S$ is a covering of the product $B \times B'$, and both surfaces have the same irregularity.
\end{thm}


\section{The Main Theorem}

\label{sect-main-thm}

This section is devoted to the proof of Theorem \ref{main-thm}. Recall that $f: S \ra B$ denotes a fibration of genus $g \geq 2$, relative irregularity $q_f$ and Clifford index $c_f$. The aim is to prove that if $f$ is not isotrivial, then $q_f \leq g-c_f$.

\begin{proof}[Proof of Theorem \ref{main-thm}]
Suppose, looking for a contradiction, that $f: S \ra B$ is non-isotrivial and that $q_f > g - c_f$. In particular, since $c_f \leq \left\lfloor\frac{g-1}{2}\right\rfloor$, we have $q_f > \frac{g+1}{2}$. Hence we can apply Theorem \ref{thm-supp-div} and assume, after a change of base, that $f$ still satisfies $q_f > g - c_f$ and is supported on a divisor $D \subset S$ such that $D \cdot C < 2g-2$ for any fibre $C$. Note that the inequality $q_f > \frac{g+1}{2}$ combined with $g \geq q_f$ implies that $g \geq 2$.

We consider now two cases:

\begin{enumerate}
\item[Case 1:] The divisor $D$ is relatively rigid, that is $h^0\left(C,\caO_C\left(D\right)\right) = 1$ for some smooth fibre $C=C_b$. In this case, after a further base change, we can assume that the fibration is stable and the divisor $D$ satisfies
$$D \cdot C \leq 2g\left(C\right)-2-C^2$$
for any component $C$ of a fibre. We can now apply Theorem \ref{thm-struc} to obtain a new fibration $h: S \ra B'$ over a curve of genus $g\left(B'\right) = q_f$. Let $\phi: C \ra B'$ be the restriction of $h$ to the smooth fibre $C$. Applying Riemann-Hurwitz we obtain
\beqn
2g-2 \geq \deg\phi \, (2q_f-2).
\enqn
At the beginning of the proof we obtained that $q_f > \frac{g+1}{2}$, so $2q_f-2 > g-1$ and thus
\beqn
2(g-1) > \deg\phi \, (g-1).
\enqn
It follows that $\deg\phi=1$, hence smooth fibre is isomorphic to $B'$ and $f$ is isotrivial.

\item[Case 2:] The divisor $D$ moves on any smooth fibre, i.e. $h^0\left(C_b,\caO_{C_b}\left(D\right)\right) \geq 2$ for every regular value $b \in B$. 

After a further change of base, we may assume that $D$ consists of $d$ sections of $f$ (possibly with multiplicities), and the new fibration is still supported on $D$. Then we can replace $D$ by a minimal subdivisor $D' \leq D$ such that $\xi$ is still supported on $D'$. Since the components of $D$ are sections of $f$, this implies that for general $b \in B$, the deformation $\xi_b$ is {\em minimally} supported on $D_{|C_b}$. Note that this might not be true if the supporting divisors were not a union of sections, as different points of $D_{|C_b}$ lying on the same irreducible component of $D$ may be redundant.

If this new $D$ is rigid on the general fibres, the proof finishes as in Case 1. Otherwise, if it still holds $h^0\left(C_b,\caO_{C_b}\left(D\right)\right) \geq 2$ for general $b \in B$, we may use Theorem \ref{thm-supp-lower-bound} to obtain
\beq \label{ineq-final-proof}
\rk \xi_b \geq \Cliff\left(D_{|C_b}\right) = c_f.
\enq
Since $V_f \subseteq \ker \partial_{\xi_b} = K_{\xi_b}$, then
\beqn
q_f = \dim V_f \leq \dim K_{\xi_b} = g - \rk \xi_b \leq g-c_f,
\enqn
contradicting our very first hypothesis.
\end{enumerate}
\end{proof}

\begin{rmk}
Note that, whenever there exists a relatively rigid divisor $D$ supporting the fibration, the inequality $q_f > \frac{g+1}{2}$ is sufficient to prove that the fibration $f$ is isotrivial (together with the structure Theorem \ref{thm-struc}), while the stronger inequality $q_f > g - c_f$ is used only if there is no such a $D$ (even allowing arbitrary base changes). Hence, all possible counterexamples to Xiao's original conjecture must fall into this second case.
\end{rmk}

\begin{rmk}
Note that the proof of the second case is indeed strictly infinitesimal, which might be a reason for the inequality to be weaker than expected.
\end{rmk}

\begin{rmk}
Although in general our bound is better than the general one (\ref{bound-5/6}) proved by Xiao, for small $c_f$ our Theorem is worse. As a extremal case, if the general fibres are hyperelliptic, $c_f=0$ and Theorem \ref{main-thm} has no content at all. But in this special case, the strong inequality $q_f \leq \frac{g+1}{2}$ was proved by Cai in \cite{Cai} using some results of Pirola \cite{Pir-Kum} about rigidity of rational curves on Kummer varieties. The same inequality has been recently proved, with very different methods, by Lu and Zuo in \cite{Lu-Zuo}.
\end{rmk}


\section{A local approach}

\label{sect-local}

The proof of Theorem \ref{thm-struc} in \cite{Xiao-1} relies on the classical Castelnuovo-de Franchis result on fibrations, for which the compactness of the surface is crucial. We present here a different way, with more local flavor, to deal with the rigid case in the proof of Theorem \ref{main-thm}. We devote this last section to the study of this different approach, which uses the theory of adjoint images and the Volumetric Theorem to be recalled now. Incidentally, the adjoint images were already a fundamental tool in the proof of Theorem \ref{thm-supp-div}, which gives a further reason to include here a short review of them. Although the theory can be developed for varieties of any dimension (see for example the work of Pirola an Zucconi \cite{PZ}), we will recall only the simplest case of curves, in which Collino and Pirola \cite{ColPir} used them for the first time and which is enough for our objective.

Let $C$ be a smooth curve of genus $g \geq 2$, and $0 \neq \xi \in H^1\left(C,T_C\right)$ a non-trivial first order deformation, corresponding to the extension
\beqn
0 \longra \caO_C \longra \caE \longra \omega_C \longra 0.
\enqn
Suppose that
$$K_{\xi} = \ker\left(H^0\left(C,\omega_C\right) \stackrel{\partial_{\xi} = \cup \, \xi}{\longra} H^1\left(C,\caO_C\right)\right) = \im\left(H^0\left(C,\caE\right)\longra H^0\left(C,\omega_C\right)\right)$$
has dimension at least 2, and let $\eta_1,\eta_2 \in K_{\xi}$ be two linearly independent 1-forms. Take $s_i \in H^0\left(C,\caE\right)$ arbitrary preimages of the $\eta_i$, and let $w \in H^0\left(C,\omega_C\right)$ be the 1-form corresponding to $s_1 \wedge s_2$ by the natural isomorphism $\bigwedge^2 \caE \cong \omega_C$. The class $[w]$ of $w$ modulo the span $W$ of $\left\{\eta_1,\eta_2\right\}$ is well-defined, independently of the choice of the preimages $s_i$.

\begin{df}
The class $[w] \in H^0\left(C,\omega_C\right) / W$ is the {\em adjoint class} of $\left\{\eta_1,\eta_2\right\}$.
\end{df}

Changing $\left\{\eta_1,\eta_2\right\}$ by another basis of $W$ amounts to multiply $[w]$ by the determinant of the change of basis. Therefore, whether $[w]$ vanishes or not is an intrinsical property of the subspace $W$, and not only of the chosen basis. Moreover, the Adjoint Theorem (\cite{ColPir} Th. 1.1.8) says that if $[w]=0$, then the deformation $\xi$ is supported on the base divisor of the linear system $\left|W\right| \subseteq \left|\omega_C\right|$. This is the main result used in the proof of Theorem \ref{thm-supp-div} to obtain supporting divisors on each fibre.

\begin{rmk} \label{rmk-ramif-alb}
In general, not much can be said of an arbitrary supporting divisor $D$, but those provided by the Adjoint Theorem enjoy some special geometric properties. Assume for example that there is a morphism $\phi: C \ra A$ from the curve to an Abelian variety $A$ such that $W \subseteq \phi^*H^0\left(A,\Omega_A^1\right)$, which is indeed the case if $C$ is a fibre of $f: S \ra B$, $A=\ker\left(\Alb\left(S\right)\ra J\left(B\right)\right)$, and $W$ is generated by sections coming from 1-forms on $S$ (see Remark \ref{rmk-kernel}). Then the base divisor of $W$, and in particular any supporting divisor deduced from the Adjoint Theorem, contains the ramification divisor of $\phi$.
\end{rmk}

We will use another result about adjoint images: the Volumetric Theorem, which we introduce now in the case of a family of curves. Let $\pi: \caC \ra U$ be a smooth family of curves over an open disc $U$, and for every $u \in U$, let $\xi_u$ be the induced first order defomation of the fibre $C_u$. Let $A$ be an Abelian variety, and let $\Phi: \caC \ra A \times U$ be a morphism such that $p_2 \circ \Phi = \pi$ (where $p_2$ denotes the second projection of the product $A \times U$), that is, a family of morphisms $\phi_u : C_u \ra A$ from the fibres of $\pi$ onto a fixed Abelian variety $A$. Given a 2-dimensional subspace $W \subseteq H^0\left(A,\Omega_A^1\right)$, denote by $W_u = \phi_u^*W \subseteq H^0\left(C_u,\omega_{C_u}\right)$ its pull-back to $C_u$. Since the elements of $W_u$ extend to all the fibres by construction, $W_u$ is contained in the kernel of $\partial_{\xi_u}$, so it is possible to define $[w_u]$, the adjoint class of $W_u$ corresponding to some chosen basis of $W$.

\begin{thm}[Volumetric Theorem(\cite{PZ}, Theorem 1.5.3)] \label{thm-vol}
Keeping the above notations, assume that $\pi$ is not isotrivial. Suppose also that for some $u_0 \in U$, $\phi_{u_0} : C_{u_0} \ra A$ is birational onto its image $Y_{u_0}$, and that $Y_{u_0}$ generates $A$ as a group. Then, for general 2-dimensional $W \subseteq H^0\left(A,\Omega_A^1\right)$ and general $u \in U$, the adjoint class $[w_u]$ is non-zero.
\end{thm}

It is worth to note that the construction of the adjoint image is very similar to the construction used by Xiao to prove the inequality (\ref{main-ineq}). In fact, the Volumetric Theorem for curves (its original statement admits fibres of any dimension) resembles the Lemma in \cite{Xiao-P1}. Furthermore, the proof of this Lemma could be adapted to prove the Volumetric Theorem as stated here. However, since the latter admits any base curve (even just an open disk) and not only $\mP^1$, there would appear some technicalities to be solved.

We will now present Proposition \ref{pro-altern}, which gives the announced alternative proof of Case 1 in the proof of Theorem \ref{main-thm}. This Proposition uses the Volumetric Theorem \ref{thm-vol} instead of Theorem \ref{thm-struc}, and hence applies for non-necessarily compact families. Note also that, because of the above discussion, the use of the Volumetric Theorem could be avoided by adapting Xiao's argument.

\begin{pro} \label{pro-altern}
Suppose that $f: S \ra B$ is a fibration where the base $B$ is a smooth, not necessarily compact curve. Assume that there is an Abelian variety $A$ of dimension $a$, and a morphism $\Phi: S \ra A \times B$ respecting the fibres of $f$ and such that the image of any restriction to a fibre $\phi_b: C_b \ra A$ generates $A$. Suppose also that the deformation is supported on a divisor $D \subset S$ such that $h^0\left(C_b,\caO_{C_b}\left(D_{|C_b}\right)\right)=1$ for general $b \in B$. If $a > \frac{g+1}{2}$, then $f$ is isotrivial.
\end{pro}

\begin{rmk} \label{rmk-global-local}
As was already sketched in Remark \ref{rmk-kernel}, if the base $B$ is compact, we may take $A$ to be the kernel of the map induced between the Albanese varieties $a_f: \Alb\left(S\right) \ra J\left(B\right)$, which has dimension $a=q_f$. Indeed, the Albanese image of $S$ is contained in $a_f^{-1}\left(B\right)$, which is a fibre bundle over $B$ with fibre $A$ and can be trivialized after replacing $B$ by an open disk. The Albanese map induces then a morphism $\Phi$ as in Proposition \ref{pro-altern}, which gives indeed a new proof of the first case in the proof of Theorem \ref{main-thm} above.
\end{rmk}

\begin{rmk}
Moreover, if $f$ admits a section, then there exists a global trivialization $a_f^{-1}\left(B\right) \cong A \times B$. In this case, the Albanese map composed with the projection to $A$ gives a map $S \ra A$, whose ramification divisor is contained in any supporting divisor as in Theorem \ref{thm-supp-div}.
\end{rmk}

\begin{proof}[Proof of Proposition \ref{pro-altern}]
Take any $b \in B$ such that $C_b$ is smooth, and let $\widetilde{C_b}$ be the image of $\phi_b: C_b \ra A$. Since $\widetilde{C_b}$ generates $A$, it has genus $g' \geq \dim A = a > \frac{g+1}{2}$. This implies, by Riemann-Hurwitz, that $\phi_b$ is birational onto its image for any regular value $b \in B$.

If $f$ is not isotrivial, the Volumetric Theorem \ref{thm-vol} implies that, for a general fibre $C=C_b$, the adjoint class of a generic 2-dimensional subspace
\beqn
W \subseteq V:= H^0\left(A,\Omega_A^1\right) \subseteq H^0\left(C,\omega_C\right)
\enqn
is non-zero. In the case $B$ is compact, then $V$ coincides with the space $V_f$ appearing in the preceding sections (see Remarks \ref{rmk-kernel} and \ref{rmk-global-local}).

However, we will now show that, for {\em every} fibre, the adjoint class of {\em every} 2-dimensional subspace of $V$ vanishes, which finishes the proof. Fix any regular value $b \in B$ and denote by $C=C_b$ the corresponding fibre, by $\xi = \xi_b$ the infinitesimal deformation induced by $f$, and by $D = D_{|C}$ the restriction of the global divisor. Let also $K=K_{\xi}$ be the kernel of $\partial_{\xi}$. Since $\xi$ is supported on $D$, Lemma \ref{lem-supp-inclu} gives the inclusion $H^0\left(C,\omega_C\left(-D\right)\right) \subseteq K$, which is in fact an equality. Indeed, on the one hand we have
\beqn
\dim H^0\left(C,\omega_C\left(-D\right)\right) = g-\deg D
\enqn
because $D$ is rigid, while on the other hand it holds
\beqn
\dim K = g-\rk\xi = g-\deg D
\enqn
because of Lemma \ref{lem-supp-inclu} and Theorem \ref{thm-supp-lower-bound}. Therefore, $V \subseteq K = H^0\left(C,\omega_C\left(-D\right)\right)$, as claimed. Note that, in order to apply Theorem \ref{thm-supp-lower-bound}, we need that $\xi$ is minimally supported on $D$. If this were not the case, we can reduce to it after a base change, as in the second case of the proof of Theorem \ref{main-thm}. Clearly this base change does not affect the isotriviality of $f$, and naturally induces a new morphism $\Phi$ with the same properties, without changing the Abelian variety $A$.

Now, since $\xi$ is supported on $D$, the upper sequence in (\ref{diag-split})
is split, giving a lifting $\omega_C\left(-D\right) \hra \Omega_{S|C}^1$ such that every pair of elements of $H^0\left(C,\omega_C\left(-D\right)\right) \subseteq H^0\left(C,\Omega_{S|C}^1\right)$ wedge to zero (they are sections of the same sub-line bundle of $\Omega_{S|C}^1$), which completes the proof.
\end{proof}

\begin{rmk}
In the above proof, to show that the images $\widetilde{C_b}$ are all isomorphic it is only necessary to use the Volumetric Theorem \ref{thm-vol}. The inequality $a > \frac{g+1}{2}$ is only used, combined with Riemann-Hurwitz, to show that the maps $\phi_b$ are birational. Therefore, if we drop the inequality $a > \frac{g+1}{2}$ from the hypothesis (but still keep that the deformations are supported on rigid divisors), the same proof shows that the fibres $C_b$ are coverings of a fixed curve $\widetilde{C_b}$.
\end{rmk}



\bigskip

\noindent Miguel \'Angel Barja

Departament de Matem\`atica Aplicada I, Universitat Polit\`ecnica de Catalunya (UPC-BarcelonaTECH)

Av. Diagonal 647, 08028 Barcelona, Spain

{\bf miguel.angel.barja@upc.edu}

\medskip

\noindent V\'{\i}ctor Gonz\'{a}lez-Alonso

Institut f\"ur Algebraische Geometrie, Leibniz Universit\"at Hannover

Welfengarten 1, 30167 Hannover, Germany

{\bf gonzalez@math.uni-hannover.de}

\medskip

\noindent Juan Carlos Naranjo

Departament d'\`Algebra i Geometria, Universitat de Barcelona

Gran Via de les Corts Catalanes 585, 08007 Barcelona, Spain

{\bf jcnaranjo@ub.edu}

\end{document}